\documentclass[12pt, reqno]{amsart}
\usepackage[utf8]{inputenc}
\usepackage[margin=1in]{geometry}
\usepackage{amsmath} 
\usepackage{amssymb} 
\usepackage{amsthm}
\usepackage{pgfplots}
\usepackage{multirow}
\pgfplotsset{compat=newest}
\usepackage[colorlinks,anchorcolor=blue,citecolor=blue,linkcolor=blue,urlcolor =blue,bookmarksdepth=2]{hyperref}

\begin{document}

\newtheorem{thm}{Theorem}[section]
\newtheorem{prop}[thm]{Proposition}
\newtheorem{lem}[thm]{Lemma}
\newtheorem{corallary}[thm]{Corallary}
\newtheorem{conj}[thm]{Conjecture}
\newtheorem{fact}{Facts}[section]
\newtheorem*{conj*}{Conjecture}

\newcommand{\ch}{\ensuremath{\mathbb{H}^2_\mathbb{C}}}
\newcommand{\rh}{\ensuremath{\mathbb{H}^2_\mathbb{R}}}
\newcommand{\C}{\mathbb{C}}
\newcommand{\Z}{\mathbb{Z}}
\newcommand{\R}{\mathbb{R}}
\newcommand{\ga}{\alpha}
\newcommand{\gb}{\beta}
\newcommand{\gc}{\gamma}

	\title{Classification of complex hyperbolic triangle groups by types}
\author{Yuhan Wang}

\newcommand{\ContactInfo}{{
\bigskip\footnotesize

\bigskip
\noindent Yuhan Wang,
\textsc{Department of Mathematics\\
Brown University\\
Providence, RI 02912, USA}\par\nopagebreak
\noindent\texttt{ywang@math.brown.edu}
}}

    \begin{abstract}
    We give a complete classification of complex hyperbolic $(n_1,n_2,n_3)$-triangle groups by types defined according to the ellipticity of two particular words of short length. This improves the Schwartz conjecture Grossi proved in \cite{grossi}. 
\end{abstract}

\maketitle

    	\section{Introduction}
    	
        A complex hyperbolic triangle is a triple of complex geodesics $C_1, C_2, C_3$ in the complex hyperbolic space $\ch$. Suppose $C_{k-1}$ and $C_{k+1}$ meet at angle $\pi/n_k$, where $n_k$ is an integer with $n_i\geq 3$ or infinity, in which case the angle is $0$. Then we call it a $(n_1,n_2,n_3)$-triangle. A complex hyperbolic $(n_1,n_2,n_3)$-triangle group is an isometry subgroup of $PU(2,1)$ generated by complex reflections in the sides of the triangle. It turns out, for each triple $(n_1,n_2,n_3)$, there is a one real parameter family of non-conjugate complex hyperbolic triangle groups. This deformation space is a half-open interval.
        
        From now on, assume $n_1\leq n_2\leq n_3$. We denote the generators by $I_1,I_2,I_3$ where $I_k$ is the complex reflection in the side of the triangle opposite $\pi/n_k$. 
        
        A natural question to ask is when are complex hyperbolic triangle groups discrete. \cite{icm} gives a complete conjectural picture of the problem. In particular, discreteness seems to be determined by words of short length. 
        Let\[
        W_A=I_1I_3I_2I_3 \text{ and } W_B=I_1I_2I_3.
        \]
        \begin{conj*}
        A $(n_1,n_2,n_3)$-complex hyperbolic triangle group is a discrete embedding in $PU(2,1)$ if neither $W_A$ nor $W_B$ is non-elliptic. 
        \end{conj*}While this conjecture is still open, it has been proven for ideal triangle groups in \cite{gp}, \cite{id1} and \cite{id2}, $(n_1,n_2,n_3)$-triangle groups with sufficiently large $n_1$ in \cite{dehn} and more recently for $(3,3,n)$-triangle groups in \cite{33n}.
        
        On the other hand, it becomes relevant to study the which one of the two words becomes elliptic first. We say $(n_1,n_2,n_3)$ has \textbf{type A} if $W_A$ becomes elliptic before $W_B$. Otherwise we say $(n_1,n_2,n_3)$ has \textbf{type B}. 
        
        It is first conjectured in \cite{icm} and proved in \cite{grossi} that $(n_1,n_2,n_3)$ has type A for $n_1<10$ and type B for $n_1>13$. In this paper, we extend this result and give a complete classification of complex hyperbolic triangle groups in terms of types as defined above. We show that the type of $(n_1, n_2, n_3)$ is determined by a polynomial in the cosines of the angles of the triangle with integer coefficients:
        
        \begin{thm}
       \label{c1}
       Let $a=4\cos^2 \pi/n_1$, $b=4\cos^2 \pi/n_2$ and $c=4\cos^2 \pi/n_3$. 
       $(n_1,n_2,n_3)$ has type A if and only if $F(a,b,c)>0$ where
       \begin{equation}
            \begin{split}
            \label{f}
           F(a,b,c)&=-176 + 96 a - 8 a^2 + 4 a^3 + a^4 + 96 b + 8 a b - 36 a^2 b + 
 2 a^3 b - 2 a^4 b - 8 b^2\\&- 36 a b^2 + 23 a^2 b^2 + a^4 b^2 + 
 4 b^3 + 2 a b^3 - 2 a^3 b^3 + b^4 - 2 a b^4 + a^2 b^4 + 120 c - 
 64 a c \\&+ 10 a^2 c + 2 a^3 c - 64 b c + 50 a b c - 14 a^2 b c - 
 2 a^3 b c + 10 b^2 c- 14 a b^2 c + 8 a^2 b^2 c \\&+ 2 b^3 c - 
 2 a b^3 c - 35 c^2 + 14 a c^2 + a^2 c^2 + 14 b c^2 - 10 a b c^2 + 
 b^2 c^2 + 4 c^3.
        \end{split}
       \end{equation}
       \end{thm}
       
       Using some properties of this polynomial, we also show the following: 
       
        \begin{corallary}\label{in/decrease}
       Suppose $(n_1,n_2,n_3)$ has type A. If $n_1'\leq n_1$, $n_2'\leq n_2$ and $n_3'\geq n_3$, then $(n_1',n_2',n_3')$ also has type A.
       \end{corallary}
       
       \begin{table}[h]
           \centering
           \begin{tabular}{|c|c|c|}
           \hline
                $(n_1,n_2,n_3)$& Output of $F$ & Type  \\
                \hline
                (3,3,10)& 114.048 & A\\
                (8,14,100) & 1.47849 & A\\
                (9,14,15) & 0.174308 & A\\
                (9,14,100) & 0.708976 & A\\
                (9,15,15) & 0.114194 & A\\
                (9,50,100) & 0.0401673 & A\\
                (14,14,14) &-0.0446055& B\\
                (15,17,30) &-0.0928291&B\\
                (20,30,50)& -0.0359106 & B\\
                (100,200,4000) & -0.0000616233 & B\\
                \hline
           \end{tabular}
           \caption{Outputs of the polynomial $F$ and types for some values of $(n_1,n_2,n_3)$.}
           \label{samples}
       \end{table}
       
       Table~\ref{samples} shows the numerical outputs under $F$ and types for some choices of $(n_1,n_2,n_3)$. They agree with the result for $n_1<10$ and $n_1>13$. They are also consistent with the properties of the parital derivatives of $F$ we are going to prove. In addition, we have a Java program, which verifies numerically the types predicted by Theorem~\ref{c1} for large numbers of $(n_1,n_2,n_3)$.
       
       Let us briefly talk about the idea of the proof of the theorem. A certain discriminant function by Goldman determines if $A\in SU(2,1)$ is elliptic by its trace. This allows us to reduce the classification problem of triangle groups to polynomial inequalities. Then the theorem essentially says these inequalities can be combined into the one given. This is a quantifier elimination problem over $\R$ and an effective algorithm called cylindrical algebraic decomposition (CAD) was first developed by Collins and later improved by others. (See \cite{cad} for details of the algorithm.) If one trust the CAD algorithm implemented by Mathematica, the theorems can be proved by a few lines of code. 
       
       However, in this paper we present a proof that can be checked by hand in principle. We first make the observation that the polynomials have even degrees for variables in the cosines of the angles of the triangle, which allows us to cut the degrees by half by using the squares of the cosines. The other key idea of our proof is dimension reduction, which is very similar to the projection step of the CAD algorithm. Namely we consider a multivariate polynomial in an inequality as a single-variate polynomial and reduce the original inequality into several inequalities of polynomials with one less variable. We would like to remark that these two things seem to be crucial from the computational perspective as the Mathematica CAD algorithm does not even seem to halt when the polynomials have four variables. 
       
       We believe that the use of CAD algorithm together with the idea we use here can help one prove stronger discrete and non-discrete results such as \cite{pqrn}, as we can use the {J}\o rgensen inequality in its full power. 
       
       The paper is organized as follows. In section 2, we first introduce the relevant basics about complex hyperbolic geometry and triangle groups. We prove our main theorem in section 3. Then in section 4, we show some properties of the type discriminant function and use them to prove the main corollary. We also include a table of type A triples in the appendix. Lastly in section 5, we prove all the inequality results we used in section 3 and 4. 
       
       \subsection*{Acknowledgements} I would like to thank my advisor Rich Schwartz for encouragement and many helpful suggestions about presenting the proofs in the paper.
       
              \section{Background}
       In this section, we introduce some necessary background about complex hyperbolic geometry and complex hyperbolic triangle groups. More detailed references are \cite{chg} and \cite{pra}.
       \subsection{Complex Hyperbolic Plane}
       
       Let $\C^{2,1}$ be $\C^3$ equipped with the Hermitian form \[
\langle u,v \rangle=u_1\overline{v_1}+u_2\overline{v_2}-u_3\overline{v_3}.
\]
Define \[
V_-=\{v\in\C^{2,1}: \langle v,v\rangle <0\}, \,\,V_0=\{v\in\C^{2,1}: \langle v,v\rangle =0\}.
\]
Consider the projection map $\mathbb{P}: \C^{2,1}\setminus\{0\}\longrightarrow \C\mathbb{P}^2$. Define \[\ch=\mathbb{P}V_-, \quad \partial \ch=\mathbb{P}V_0.\]
$\ch$ is called the complex hyperbolic plane with ideal boundary $\partial\ch$. By the inclusion from $\C^2$ to $\C\mathbb{P}^2$, $(z,w)\mapsto [z:w:1]$, we can identify $\ch$ with the unit ball in $\C^2$ and $\partial \ch$ with the unit sphere $S^3$. 
 \[\ch=\{(z,w): |z|^2+|w|^2< 1\}, \quad \partial \ch=\{(z,w): |z|^2+|w|^2= 1\}.\]

\subsection{Distance and Isometry}

For $u,v\in H^2_\C$, let $\mathbf{u}$ and $\mathbf{v}$ be their lifts in $\C^{2,1}$. Then the distance $d (u,v)$ between $u$ and $v$ is given by \[
\cosh^2\left(\frac{d (u,v)}{2}\right)=\frac{\langle u,v\rangle\langle v,u\rangle}{\langle u,u\rangle\langle v,v\rangle}.
\]
$SU(2,1)$ is the group of matrices with determinant 1 preserving the Hermitian form. Its projectivization $PU(2,1)$ acts isometrically on $H^2_\C$. 
In fact, the group of holomorphic isometries is exactly $PU(2,1)$ and the full group of isometries is generated by $PU(2,1)$ and the anti-holomorphic map $(z,w)\mapsto (\overline{z}, \overline{w})$.

\subsection{Classification of Isometries by fixed points}
  A holomorphic complex hyperbolic isometry $M$ is said to be:
  \begin{enumerate}
  \item loxodromic if it fixes exactly two points of $\partial \ch$
  \item parabolic if it fixes exactly one point of $\partial \ch$
  \item elliptic if it fixes at least one point of $\ch$
  \end{enumerate}
  
  The following beautiful theorem due to Goldman \cite{chg} allows one to classify isometries in terms of the types defined above by using the traces of their matrix representations in $SU(2,1)$. This is the building block of our approach to the problem.
  \begin{thm}\label{disc}
       Let $f:\mathbb{C}\rightarrow\mathbb{R}$ be the discriminant function by\[
       f(z)=|z|^4-8\operatorname{Re}(z^3)+18|z|^2-27.
       \]
       Suppose $M$ is an element of $SU(2,1)$ with trace $\tau$. Then \begin{enumerate}
       \item $M$ is loxodromic if and only if $f(\tau)>0$.
       \item $M$ is regular elliptic if and only if $f(\tau)<0$.
       \item $M$ has a repeated eigenvalue if and only if $f(\tau)=0$.
       \end{enumerate}
       If $\tau\in\mathbb{R}$, then $M$ is elliptic if and only if $\tau\in[-1,3)$; parabolic if and only if $\tau=3$.
       \end{thm}

       \subsection{Complex Reflections}
 There are two kinds of totally geodesic 2-dimensional submanifolds of the complex hyperbolic plane:
\begin{itemize}
\item Totally real subspaces (real slices), which is isometric to $\ch\cap \mathbb{R}^2$.
\item Complex geodesics (complex slices), which is isometric to $\ch\cap \C$. 
\end{itemize}
A complex reflection is a holomorphic isometry conjugate to the map $(z,w)\mapsto (z,-w)$. 
Let $\mathbf{c}\in\C^{2,1}$ such that $\langle \mathbf{c}, \mathbf{c}\rangle>0$. Then for any $\mathbf{v}\in\C^{2,1}$, we define the complex reflection with polar vector $\mathbf{c}$:\[
I_\mathbf{c}(\mathbf{v})=-\mathbf{z}+\frac{2\langle \mathbf{v},\mathbf{c}\rangle}{\langle \mathbf{c}, \mathbf{c}\rangle}\mathbf{c}
\] 
The fixed point set of a complex reflection is a complex slice. So a  polar vector $\mathbf{c}$ determines a complex slice. Conversely a complex slice determines a polar vector $\mathbf{c}$ up to multiplication by a scalar.

Two complex slices with polar vectors $c_i$ normalized to have $\langle c_i,c_i\rangle=1$ for $i=1,2$ intersect at angle $\alpha$ if and only if $|\langle c_1,c_2\rangle|=\cos{\alpha}$.

  \subsection{Complex Hyperbolic Triangle Groups}
  All the information in this subsection comes from \cite{pra}.
   A complex hyperbolic triangle is a triple of complex geodesics $C_1, C_2, C_3$ in the complex hyperbolic space $\ch$. Suppose $C_{k-1}$ and $C_{k+1}$ meet at angle $\pi/n_k$, where $n_k$ is an integer with $n_i\geq 3$ or $\infty$, in which case the angle is $0$. Then we call it a $(n_1,n_2,n_3)$-triangle. A complex hyperbolic $(n_1,n_2,n_3)$-triangle group is an isometry subgroup of $PU(2,1)$ generated by $I_1,I_2,I_3$, complex reflections in the sides of the triangle. 
   
   Given any triple $(n_1,n_2,n_3)$, there is a one real parameter family of $(n_1,n_2,n_3)$-triangle groups parametrized by the angular invariant of the polar vectors of the sides:\[
   \theta=\arg\left(\frac{\prod_{k=1}^3\langle c_{k-1}, c_{k+1}\rangle}{\prod_{k=1}^3\langle c_{k}, c_{k}\rangle}\right).
   \] Let $r_k=\cos{\pi/n_k}$ for $k=1,2,3$. We use the following proposition to compute the parameter space. 
   \begin{prop}\label{para}
   Let $C_1,C_2,C_3$ be three complex slices. Suppose $C_{k-1}$ and $C_{k+1}$ intersect at angle $\pi/n_k$ for $k=1,2,3$. Then they form a complex hyperbolic triangle if and only if \[
   \cos{\theta}<\frac{r_1^2+r_2^2+r_3^2-1}{2r_1r_2r_3},
   \]where $\theta$ is the angular invariant of the polar vectors of the complex slices as defined above.
   
   \end{prop}
   
   Now define $t=\cos{\theta}$ and $t_u=\min\left\{\frac{r_1^2+r_2^2+r_3^2-1}{2r_1r_2r_3},1\right\}$. The deformation space of $(n_1,n_2,n_3)$-triangle groups is a half open interval $[-1,t_u)$ parametrized by $t$. 
   
   We end this section by listing the traces of the two words of interest. Let\[
        W_A=I_1I_3I_2I_3 \text{ and } W_B=I_1I_2I_3.
        \]
        \begin{lem}\label{trace}
        The trace of $W_A$ is \[\tau_A=16r_1^2r_2^2+4r_3^2-16r_1r_2r_3\cos{\theta}.\]
        The trace of $W_B$ is \[
        \tau_B=8r_1r_2r_3e^{i\theta}-4(r_1^2+r_2^2+r_3^2)+3.
        \]
        \end{lem}

       \section{Type A and Type B}
      In this section, we prove our main result Theorem \ref{c1}. Recall we say $(n_1,n_2,n_3)$ has \textbf{type A} if $W_A$ becomes elliptic before $W_B$. Otherwise we say $(n_1,n_2,n_3)$ has \textbf{type B}. 
      
    To simplify computations, we make the following substitutions. Let $T=r_1 r_2 r_3 t$, $a=4r_1^2$, $b=4r_2^2$ and $c=4r_3^2$.
    
    Now let $\mathcal{I}_A$ be the set of values of $T$ such that $W_A$ is non-elliptic and $\mathcal{I}_B$ be that $W_B$ is non-elliptic. Let $\mathcal{I}=\mathcal{I}_A\cap\mathcal{I}_B$. We call $\mathcal{I}$ the critical interval\footnote{$(n_1,n_2,n_3)$-triangle groups are conjectured to be discrete and faithful in $SU(2,1)$ on $\mathcal{I}$ \cite{icm}. So the critical interval should be a subset of the full deformation space. This is indeed the case and is proved in \cite{grossi} by showing $\mathcal{I}_A$ is always a subset of the deformation space.}. $\mathcal{I}_A$ is a closed interval since\[
    W_A\text{ is regular-elliptic }\Longleftrightarrow f(\tau_A)<0 \Longleftrightarrow T>T_A=\frac{ab+c-4}{16},
    \]where $f$ is Goldman's discriminant in Theorem \ref{disc}. 
    
    Let $f_B(T)=f(\tau_B)$. From \eqref{f} and notice that $F(a,b,c)=f_B(T_A)$, Theorem \ref{c1} immediately follows from that $\mathcal{I}$ is indeed a closed interval. We state this as a theorem and prove it with Proposition~\ref{propmain}.

       \begin{thm}\label{interval}
       $\mathcal{I}$ is a closed interval. 
       \end{thm}
       
       \begin{proof}
       As a polynomial with variable $T$, $f_B$ is cubic with negative leading coefficient. Then it follows from Proposition~\ref{propmain} and calculus that $T_A$ must be less than or equal to the second smallest root of $f_B$ (see Figure \ref{tcube}) and therefore $\mathcal{I}$ is a closed interval.
       \end{proof}
       
       \begin{figure}
           \centering
           \begin{tikzpicture}
\begin{axis}[
axis y line=none,
hide y axis,
    axis x line=middle,
    axis line style={->},
    tick style={color=black},
    xtick=\empty,
    xlabel = $T$
]
\addplot [
    domain=-0.2:1, 
    samples=20, 
    color=black,
]
{-(x-1)*(x-2)*x};
\addplot [
    domain=-0.5:2.3, 
    samples=10, 
    color=black,
    ]
    {0};
\addplot [
    domain=1:1.57, 
    samples=10, 
    color=red,
    ]
    {-(x-1)*(x-2)*x};
\addplot [
    domain=1.57:2.1, 
    samples=10, 
    color=blue,
]
{-(x-1)*(x-2)*x};
 
\end{axis}
\end{tikzpicture}
           \caption{Suppose $f_B(T)$ has 3 roots. Statement 1 of Proposition \ref{propmain} shows $T_A$ is not on the red part of the curve, while statement 2 shows $T_A$ is not on the blue part.}
           \label{tcube}
       \end{figure}
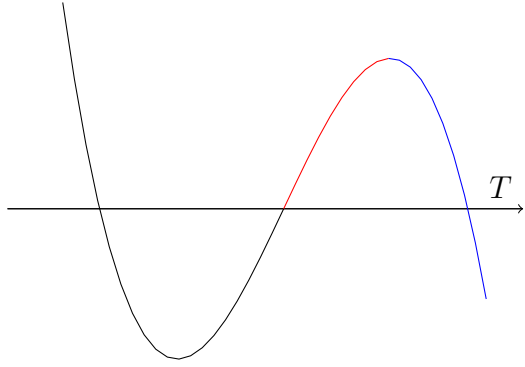

        Before we prove our main proposition, we would like to give an idea about how the proof goes. It all boils down to statements about inequalities of polynomials with three variables $a,b,c$, for which we can also think of them as quadratic polynomials in $c$. So the problem becomes comparing the zeros of the polynomials and the bounds on $c$, which we can geometrically think of as the study of the non-intersection of the corresponding semi-algebraic sets. It turns out that the relevant semi-algebraic sets usually live in different half-planes or quadrants bounded by horizontal or vertical lines, which we will prove in the last section of the paper. 
        \bigskip
       
       \begin{prop}\label{propmain}
       Let $1\leq a\leq b\leq c\leq 4$. Then
       
       \begin{enumerate}
       \item $f_B(T_A)>0\implies f'_B(T_A)<0$.
       \item $f'_B(T_A)<0\implies f''_B(T_A)>0$.
       \end{enumerate}
       \end{prop}
       \begin{proof}
       We set $g_1=\frac{\partial F}{\partial c}$, $g_2=f'_B(T_A)$ and $g_3=f''_B(T_A)$ and consider them as polynomials in $\Z[a,b][c]$.
       
       1. By Proposition~\ref{cderivative}, it suffices to show $g_1(c)>0$ implies $g_2(c)<0$.  We note that both $g_1$ and $g_2$ are quadratic and the leading coefficient of $g_1$ is positive whereas the one of $g_2$ is negative.
       
       
       
       First we claim $g_2'(b)\geq 0$ $\implies g_2(b)\geq 0$. By \ref{propmain.g2'a} and \ref{propmain.g2a} of Lemma~\ref{ablems}, $g_2'(b)\geq 0~\implies a\geq~\frac{\sqrt{134}-4}{2}$ and $g_2(b)<0\implies a<2+\sqrt{3}$. But $\frac{\sqrt{134}-4}{2}>2+\sqrt{3}$ and we have a contradiction.
       
       Now we split into two cases:
       
      (i) $g_1(b)> 0$. Since $g_2(b)<0$ implies $g_2'(b)<0$, it suffices to check $g_2(b)<0$. We give a proof by contradiction. Assume $g_2(b)\geq 0$. First note that $g_2(b)\geq 0\implies b\geq 2+\sqrt{3}$ (Lemma \ref{ablems}.\ref{propmain.g2b}). Now we want to consider both $g_1(b)$ and $g_2(b)$ as polynomials in $a$. So we set
      \[
      \phi(a)=\frac{g_1(b)}{2}=a^3 (1 - b) +60 - 67 b + 25 b^2 + 2 b^3 + a^2 (5 - 6 b + 4 b^2) + 
 a (-32 + 39 b - 17 b^2 - b^3)
      \]
      \[
      \psi(a)=\frac{g_2(b)}{32}=a^3 (-1 + b) +12 - 27 b + 2 b^2 - 4 b^3 + a^2 (-1 + b - 3 b^2) + 
 a (-24 + 23 b + b^2 + 2 b^3)
      \]
       $\phi(a)>0$ if and only if $1\leq a< a_\phi$ where $a_\phi$ is the smallest real root of $\phi$. This follows from the fact that $\phi(b)\leq 0$ and $\phi''(b)>0$ for $2+\sqrt{3}\leq  b\leq 4$. 
       On the other hand, since $\psi$ has positive leading coefficient,  $\psi(a)\geq 0$ implies $a_\psi\leq a$ where $a_\psi$ is the smallest real root of $\psi$. 
       
       Now we show $a_\phi\leq a_\psi$ and this will lead to a contradiction. It suffices to show $\psi(a_\phi)\leq 0$, $\psi'(a_\phi)>0$ and $\psi''(a_\phi)<0$. We only prove $\psi(a_\phi)\leq 0$ here. The same strategy will work for $\psi'(a_\phi)>0$ and $\psi''(a_\phi)<0$ is easy. 
       
       By polynomial long division, \begin{align*}
          \psi(a_\phi)&=-\phi(a_\phi)+(b-4)\left(a_\phi^2 (-1 + b)  + a_\phi (14 - 12 b + b^2)-18+ 19 b - 2 b^2\right)\\&=(b-4)\left(a_\phi^2 (-1 + b)  + a_\phi (14 - 12 b + b^2)-18+ 19 b - 2 b^2\right),\\
          \phi(a_\phi)&=\frac{\psi(a_\phi)}{b-4}\left( \frac{19 - 18 b + 5 b^2}{b-1}-a_\phi\right)-\frac{6\chi(a_\phi)}{b-1}=0
       \end{align*}where\[
       \chi(a)=-47 + 93 b - 63 b^2 + 18 b^3 - 2 b^4 + 
  a (b-2) (-18 + 22 b - 8 b^2 + b^3).
       \]Note that $\frac{19 - 18 b + 5 b^2}{b-1}>b\geq a_\phi$ and $-18 + 22 b - 8 b^2 + b^3>0$ for all $b\geq 2+\sqrt{3}$. Thus,  \[\psi(a_\phi)<0 \iff \chi(a_\phi)>0 \iff a_\phi>\frac{47 - 93 b + 63 b^2 - 18 b^3 + 2 b^4}{(b-2) (-18 + 22 b - 8 b^2 + b^3)}\] for which it suffices to show \[
       b>\frac{47 - 93 b + 63 b^2 - 18 b^3 + 2 b^4}{(b-2) (-18 + 22 b - 8 b^2 + b^3)} \quad\text{ and }\quad  \phi\left(\frac{47 - 93 b + 63 b^2 - 18 b^3 + 2 b^4}{(b-2) (-18 + 22 b - 8 b^2 + b^3)}\right)>0.\] It is easy to check these last two inequalities are true for $b\geq 2+\sqrt{3}$.
       
       (ii) $g_1(b)\leq 0$. Let $c_1$ and $c_2$ be the two roots of $g_1$ with $c_1\leq c_2$. It suffices to check $g_2(c_2)<0$. First we note that $c_2<4$. Otherwise, $g_1(c)\leq 0$ for all $b\leq c\leq 4$. So we must have\[
       g_1(4)=2\gamma>0
       \]where \[
        \gamma=16 - 33 a b + 6 a^2 b^2 + 24 (a + b) - 10 a b (a + b) + 9 (a + b)^2 - 
 a b (a + b)^2 + (a + b)^3.
       \]
       Now we use the polynomial long division to get\begin{align*}
       g_2(c_2)&=\frac{-8(a+b+1)g_1(c_2)}{3}+\frac{16}{3}\left(\alpha c_2-4\alpha+(a+b-5)\gamma\right)\\&=\frac{16}{3}\left(\alpha c_2-4\alpha+(a+b-5)\gamma\right).
       \end{align*}where\[
       \alpha=-53 - 6 a b (a + b-1)+ 3 (a + b)  + 3 (a + b)^2 + (a + b)^3.
       \]
        Recall $g_1$ must have non-negative discriminant and $a+b\geq 5$ (See Lemma~\ref{claim}). Since $4ab~<~(a+b)^2$, we have $\alpha>0$. Therefore, $h(c_2)<0$ if and only if $c_2<c_0$ wherewhere\[
        c_0=4-\frac{(a+b-5)\gamma}{\alpha}.
        \] Now it suffices to show $g_1\left(c_0\right)>0$ and $g_1'\left(c_0\right)>0$.  We explicitly compute\[
            g_1\left(c_0\right)=\frac{72(7 - a + a^2 - b - 2 a b + b^2)^2}{\alpha^2}\gamma>0.
         \] 
       Now we show that we cannot have both $g_1'\left(c_0\right)\leq 0$ and $\gamma>0$. 
       \[
       g_1'\left(c_0\right)=\epsilon+3 (a - b)^2\delta
       \]where\[
       \epsilon=542 + 1090 (a + b) - 73 (a + b)^2 - 97 (a + b)^3 + 
 19 (a + b)^4 - (a + b)^5
       \]and\[
       \delta= 211 - 48 a b + 39 (a + b) - 3 (a + b)^2 + (a + b)^3.
       \]By using $-ab\geq -(a+b)^2/4$, it is easy to see $\delta>0$. So \[g_1'\left(c_0\right)\leq 0\implies \epsilon\leq 0 \implies a+b> \frac{79}{10}.\]
       On the other hand, $\gamma>0\implies a+b<\frac{39}{5}$ (Lemma~\ref{ablems}.\ref{propmain.gamma}). This is a contradiction. We are done.
       
       2. We want to show $g_2(c)<0$ implies $g_3(c)>0$ for all $1\leq a\leq b\leq c\leq 4$. First we compute \[
        \frac{g_3(c)}{512}=21 + 6 a + a^2 + 6 b - 10 a b + b^2 + (-6 + 2 a + 2 b) c + c^2.
       \] So $g_3$ is quadratic with positive leading coefficient. Also, \[
       \frac{g_3'(c)}{512}=2(a+b+c-3)\geq 0.
       \]Again we split into two cases to make the proof:
       
       (i) $g_2(b)<0$. Since $g_3'(c)\geq 0$, it suffices to check $g_3(b)>0$. We prove this by contradiction.
       $g_3(b)\leq 0$ implies $a\geq 1+2\sqrt{2}$ (Lemma~\ref{ablems}.\ref{propmain.g3a}).
       On the other hand, $g_2(b)<0\implies a<2+\sqrt{3}$ (Lemma~\ref{ablems}.\ref{propmain.g2a}). So we have a contradiction.
       
       (ii) $g_2(b)\geq 0$. Let $c_1$ and $c_2$ be the two roots of $h$ with $c_1\leq c_2$. It suffices to show $g_3(c_2)>0$. By long division, we get\[
       \frac{(a+b+1)j(c_2)}{512}=-\frac{h(c_2)}{32}+(\sigma +  \lambda c)
       \]where\begin{align*}
            \lambda&=-9 + ab(a+b+3),\\
            \sigma&=33 + 9 a b - 6 a^2 b^2 + 3 (a + b) - 6 a b (a + b) + 6 (a + b)^2 + 
 a b (a + b)^2.
       \end{align*}
      Note that $g_2(b)\geq 0$ implies $b\geq 2+\sqrt{3}$ (Lemma~\ref{ablems}.\ref{propmain.g2b}). So $\lambda>0$ and $g_3(c_2)>0$ if and only if $c_2>\frac{-\sigma}{\lambda}$. Therefore it suffices to show $g_2\left(\frac{-\sigma}{\lambda}\right)>0$. To this, we compute\[
      g_2\left(\frac{-\sigma}{\lambda}\right)=\frac{6\left(a+b+1\right)\left(-4 - a - b + a b\right)\left(64\eta+(a-b)^2\kappa\right)}{\lambda^2}
      \]where\begin{align*}
          \eta&=1344 + 1152 (a + b) - 384 (a + b)^2 + 192 (a + b)^3 + 12 (a + b)^4 - 
 12 (a + b)^5 + (a + b)^6\\
 \kappa&=384 - 240 a b + 48 a^2 b^2 + 48 a b (a + b) - 12 (a + b)^2 - 
 4 a b (a + b)^2 + 12 (a + b)^3 - (a + b)^4.
      \end{align*}
      It is easy to check $\eta>0$. As $g_2(b)\implies a\geq \frac{33-\sqrt{129}}{6}$, we also have $-4-a-b+ab>0$. Now we are left to show $\kappa>0$. Note that $(a-4)(b-4)\geq 0$ and so $ab\geq 4(a+b-4)$. We can also check that $a^2b^2\geq 16(a+b-4)^2$ for $b\geq 2+\sqrt{3}$. In addition, we have $-4a b\geq -(a+b)^2$. Thus, 
      \[
      \kappa \geq 12672 - 6912 (a + b) + 888 (a + b)^2 + 12 (a + b)^3 - 2 (a + b)^4>0
      \] for $3+\sqrt{3}\leq a+b \leq 8$. We are done.
       \end{proof}
       
       \section{Properties of the function \texorpdfstring{$F$}{F}}
       In this section, we prove some properties about the partial derivatives of the function $F$ that can be used to determine precisely the types of triples $(n_1,n_2,n_3)$. The proofs are analogous to Proposition~\ref{propmain}.
       \begin{prop}\label{cderivative}
       Suppose $1\leq a\leq b\leq c\leq 4$. Then $F>0$ implies $\frac{\partial F}{\partial c}>0$.
       \end{prop}
       \begin{proof}
       Set $g=F\in\Z[a,b][c]$. After rearranging, we get
       \begin{align*}
       g(c)&=\lambda_3c^3+\lambda_2c^2+\lambda_1c+\lambda_0\\
       g'(c)&=3\lambda_3c^2+2\lambda_2c+\lambda_1
       \end{align*}where \begin{align*}
           \lambda_0&=-176 + 96 a - 8 a^2 + 4 a^3 + a^4 (-1 + b)^2 + 96 b + 8 a b - 
 36 a^2 b + 2 a^3 b \\&- 8 b^2 - 36 a b^2 + 23 a^2 b^2 + 4 b^3 + 
 2 a b^3 - 2 a^3 b^3 + b^4 - 2 a b^4 + a^2 b^4\\
 \lambda_1&=120 - 64 a + 10 a^2 + 2 a^3 - 64 b + 50 a b - 14 a^2 b - 2 a^3 b\\ &+ 
10 b^2 - 14 a b^2 + 8 a^2 b^2 + 2 b^3 - 2 a b^3\\
 \lambda_2&=-35 + 14 a + a^2 + 14 b - 10 a b + b^2\\
 \lambda_3&=4.
       \end{align*}
       Suppose $g(c)>0$.
       Since the leading coefficient of $g'$ is positive, $g'(c)>0$ when its discriminant $\Delta<0$. So we assume $\Delta\geq 0$.
       
       Let $c_1$ and $c_2$ be the two roots of $g'$. We claim $b>c_1$ and if $c_1\leq b\leq c_2$, then $g(b)\leq 0$. Then it follows that for $c\geq b$, $g(c)>0$ implies $g'(c)>0$. We rephrase the claim as Lemma~\ref{claim}.
      
       \end{proof}
       \begin{lem}\label{claim}Let $1\leq a\leq b\leq 4$. Define $g=F\in\Z[a,b][c]$. Suppose $\Delta\geq 0$ where $\Delta$ is the discriminant of $g'$. Then \begin{enumerate}
       \item $g''(b)<0\implies g'(b)\leq 0$.
       \item $g'(b)\leq 0\implies g(b)\leq 0$.
       \end{enumerate}
       \end{lem}
       \begin{proof}
       We first use the condition $\Delta\geq 0$ to get constrains on $a$ and $b$.
       \[
       \Delta=4 (-5 + a + b) (43 - 96 a b + 51 (a + b) + 9 (a + b)^2 + (a + b)^3)\geq 0.
       \]
       Using the fact $ab<(a+b)^2/4$, we get \[
       43 - 96 a b + 51 (a + b) + 9 (a + b)^2 + (a + b)^3\geq
 43 + 51 (a + b) - 15 (a + b)^2 + (a + b)^3>0
       \] for all $a+b\geq 2$. Therefore $\Delta\geq 0$ if and only if $a+b\geq 5$.
       
       1. Since $g''$ is increasing, from $g''(b)<0$ and $\frac{b+a}{2}\leq b$, we get \[
       g''\left(\frac{b+a}{2}\right)=-70 - 24 a b + 40 (a + b) + 2 (a + b)^2<0,
       \] which together with $a+b\geq 5$ implies $a+b>\frac{10+\sqrt{30}}{2}$.
       
       We rearrange $g'(b)$ to get
       \[
       g'(b)=2(b-a)h_b(a)+k(b)
       \]where\begin{align*}
           h_b(a)&=(b-1)a^2+(-5 + 5 b - 3 b^2)a+(32 - 44 b + 22 b^2 - 2 b^3)\\
           k(b)&=120 - 198 b + 138 b^2 - 40 b^3 + 4 b^4.
       \end{align*}
       It is easy to check $h_b\left(\frac{10+\sqrt{30}}{2}-b\right)<0$, $h_b(b)<0$ and $k(b)\leq 0$ for all $\frac{10+\sqrt{30}}{4}<b\leq 4$. Since $h_b$ is quadratic in $a$ with positive leading coefficient, we get $g'(b)\leq 0$.
       
       2. Suppose $g'(b)\leq 0$ and $g(b)>0$.  We show this leads to a contradiction. 
       
       $g'(b)\leq 0\implies a\geq \frac{1}{6} \left(33 - \sqrt{129}\right)$ (Lemma~\ref{ablems}.\ref{claim.g'a}) and $g(b)>0\implies a<\frac{11}{3}$ (Lemma~\ref{ablems}.\ref{claim.ga}). So we have $\frac{1}{6} \left(33 - \sqrt{129}\right)\leq a<\frac{11}{3}$. Under this, $g'(b)\leq 0\implies b>\frac{389}{100}$ (Lemma~\ref{ablems}.\ref{claim.g'b}) and $g(b)> 0\implies b<\frac{387}{100}$ (Lemma~\ref{ablems}.\ref{claim.gb}). That is a contradiction.
       \end{proof}
       \bigskip
       \begin{prop}\label{abderivative}
       Let $2\leq a\leq b\leq c\leq 4$. Suppose $F(a,b,c)>0$. Then \begin{enumerate}
           \item $\frac{\partial F}{\partial a}<0$.
           \item $\frac{\partial F}{\partial b}\leq 0$ with equality at $a=2, b=c=4$.
       \end{enumerate}
       \end{prop}
       \begin{proof}
       We set $g_1=\frac{\partial F}{\partial c}$, $g_4=\frac{\partial F}{\partial a}$ and $g_5=\frac{\partial F}{\partial b}$ and consider them as polynomials in $\Z[a,b][c]$.
       
        1. We first show $g_1(c)>0$ implies $g_4(c)<0$. Note that $g_4(c)$ is quadratic with leading coefficient $14+2a-10b$, which is negative given $2\leq a\leq b\leq 4$. $g_4$ has the property that $g_4'(b)>0$ implies $g_4(b)\geq 0$. This is because $g_4'(b)>0\implies a>\frac{15 - \sqrt{13}}{3}$ (Lemma~\ref{ablems2}.\ref{abderivative.g4'a}) and $g_4(b)< 0\implies a<\frac{377}{100}$ (Lemma~\ref{ablems2}.\ref{abderivative.g4a}), but $\frac{377}{100}<\frac{15 - \sqrt{13}}{3}$.
       
       To proceed, we split into to cases:
       
       (i) $g_1(b)> 0$. It suffices to show $g_4(b)<0$.  From $g_1(b)> 0$ and $g_4(b)\geq 0$, we get $\frac{11}{3}\leq a<\frac{372}{100}$ (Lemma~\ref{ablems2}.\ref{abderivative.g4ap} and Lemma~\ref{ablems}.\ref{abderivative.g1a}). For $a$ in this interval, the inequalities lead to $b>\frac{394}{100}$ (Lemma~\ref{ablems2}.\ref{abderivative.g4b}) and $b<\frac{390}{100}$ (Lemma~\ref{ablems}.\ref{abderivative.g1b}), and thus a contradiction.
       
       (ii) $g_1(b)\leq 0$. Let $c_1$ and $c_2$ be the two roots of $g$ with $c_1\leq c_2$. It suffices to show $g_4(c_2)<0$. We compute \[
       g_4(c_2)=\frac{(7+a-5b)g_1(c_2)}{6}+\frac{1}{3}\left(\alpha c_2-(4-4b+ab)\alpha-(a+b-5)\gamma\right)
       \]where\begin{align}
           \alpha&=53 - 3 a - 3 a^2 - a^3 - 123 b + 42 a b - 3 a^2 b + 21 b^2 - 
  3 a b^2 - b^3\\
  \gamma&= (b-a)(-24 - 9 a - a^2 + 49 b - 4 a b - 3 b^2)+4(4-b)\label{g4gamma}
       \end{align}
       
       We first show that $\alpha>0$. $g_1(b)\leq 0\implies a\geq\frac{33-\sqrt{129}}{6}$ (Lemma~\ref{ablems}.\ref{claim.g'a}). Note that in order to have a non-vacuous statement, we have $g_1(4)>0$, which implies $a<\frac{384}{100}$ (Lemma~\ref{ablems2}.\ref{abderivative.g14}). However, $\alpha\leq 0$ and $a\geq\frac{33-\sqrt{129}}{6}\implies a>\frac{388}{100}$ (Lemma~\ref{ablems2}.\ref{abderivative.alpha}).
       
       Therefore, \begin{align*}
           g_4(c_2)<0&\iff c_2<c_0\\&\iff g_1\left(c_0\right)>0\text{ and } g_1'\left(c_0\right)>0.
       \end{align*}where \[c_0=4-4b+ab+\frac{(a+b-5)\gamma}{\alpha}.\]       
        $g_1'\left(c_0\right)>0$ is done in Lemma~\ref{ablems2}.\ref{abderivative.g1'}. To show $g_1(c_0)>0$, we simply compute \[
        g_1\left(c_0\right)=\frac{72(7 - a + a^2 - b - 2 a b + b^2)^2}{\alpha^2}\gamma.
       \] Now since $b\geq a>0$, \[
       -24 - 9 a - a^2 + 49 b - 4 a b - 3 b^2\geq -24 + 40 b - 8 b^2>0
       \]for $2\leq b\leq 4$. Therefore, from equation~(\ref{g4gamma}), $\gamma\geq 0$ with equality only if $a=b=4$. However, $a<\frac{384}{100}$. So $\gamma>0$.

       2. Now we show $g_1(c)>0\implies g_5(c)\leq 0$. We first check that when $a=2,b=c=4$, $g_5(c)=0$. For the rest of the section, we assume $a>2$. $g_5(c)$ is quadratic with leading coefficient $14+2b-10a$. Suppose $14+2b-10a\geq 0$. Both $g_5(b)<0$ and $g_5(4)<0$, since by Lemma~\ref{ablems2}.\ref{abderivative.g5b} and Lemma~\ref{ablems2}.\ref{abderivative.g54}, $g_5(b)\geq 0$ and $g_5(4)\geq 0$ both imply $ a\geq \frac{33-\sqrt{129}}{6}$, but this contradicts $14+2b-10a\geq 0$. So for the rest of the proof, we assume $g_5$ has negative leading coefficient. We have that $g_5'(b)>0\implies g_5(b)\geq 0$ by Lemma~\ref{ablems2}.\ref{abderivative.g5bn} and Lemma~\ref{ablems2}.\ref{abderivative.g5'b}. Now we have two cases:
       
       (i) $g_1(b)>0$. We want to show $g_5(b)<0$. This is Lemma~\ref{2d}.\ref{g1pg5n}.

       (ii) $g_1(b)\leq 0$. Let $c_1$ and $c_2$ be the two roots of $g_1$ with $c_1\leq c_2$. We want to show $g_5(c_2)<0$. We compute \[
        g_5(c_2)=\frac{(7+b-5a)g_1(c_2)}{6}+\frac{1}{3}\left(\alpha c_2-(4-4a+ab)\alpha-(a+b-5)\gamma\right)
       \]where\begin{align*}
           \alpha&=53 - 123 a + 
 21 a^2 - a^3 + (-3 + 42 a - 3 a^2) b + (-3 - 3 a) b^2 - b^3\\
  \gamma&= 16 - 28 a + 49 a^2 - 3 a^3 + (24 - 58 a - a^2) b + (9 + 3 a) b^2 + b^3.
  \end{align*}Before we split into more subcases, we note that $a\geq \frac{33-\sqrt{129}}{6}$ from $g_1(b)\leq 0$. 
  
  (a) $\alpha=0$. So to $g_5(c_2)<0$ is equivalent ot show $\gamma>0$.
  
  Now we compute \begin{equation*}
      g_1\left(c_0\right)=\frac{72(7 - a + a^2 - b - 2 a b + b^2)^2}{\alpha^2}\gamma, \text{ where } c_0=4-4a+ab+\frac{a+b-5}{\alpha}\gamma.
  \end{equation*}
  
  (b) $\alpha>0$. In this case, $g_5(c_2)<0$ is equivalent to $c_2<c_0$. In turn, this is equivalent to the following $g_1\left(c_0\right)>0$ and $g_1'\left(c_0\right)>0$.
  
  (c) $\alpha<0$. In this case, $g_5(c_2)<0$ if and only if $c_2>c_0$, which is equivalent to the following: $g_1\left(c_0\right)>0$ implies $g_1'\left(c_0\right)<0$.
  
  They follow from the following:\[
  \alpha\geq 0 \implies 
 \gamma>0\text{, and } \gamma>0\implies \alpha g_1'\left(c_0\right)>0.\] They are proved in Lemma~\ref{2d}.\ref{alphagamma} and Lemma~\ref{2d}.\ref{gammag1'}.
  
       \end{proof}
       
       Proposition~\ref{abderivative} does not apply when $a=1$, but it is easy to check $F(1,1,2)>0$ and $F(1,n_2,n_2)>0$ for all $2\leq n_2\leq 4$. So combining the two propositions with Theorem~\ref{c1}, we get Corollary~\ref{in/decrease}. 
       \section{Inequalities in two variables}
       This last section consists of three lemmas about polynomial inequalities in two variables. The first two lemmas show that various semi-algebraic sets lie on some half-space or quadrant bounded by horizontal and/or vertical lines. We give proofs to some of them. The rest is analogous to at least one of them. We would like to comment that whenever a bound, as a root of some polynomial, does not have a radical expression, we replace by a correct but not sharp rational bound so that everything can be calculated without precision issues. 
       \begin{lem}\label{ablems}
Let $1\leq a\leq b\leq 4$.
\begin{enumerate}
    \item\label{propmain.g2'a} $h_1(b)=-3 + 4 a - 2 a^2 + (2 - 3 a + a^2) b + (-4 + a) b^2\geq 0\implies a\geq \frac{\sqrt{134}-4}{2}$.
    \item\label{propmain.g2a} $h_2(b)=12 - 24 a - a^2 - a^3 + (-27 + 23 a + a^2 + a^3) b + (2 + a - 
    3 a^2) b^2 + (-4 + 2 a) b^3<0\implies a<2+\sqrt{3}$.
    \item\label{propmain.g2b} $j_3(a)=12 + a^3 (-1 + b) - 27 b + 2 b^2 - 4 b^3 + a^2 (-1 + b - 3 b^2) + 
 a (-24 + 23 b + b^2 + 2 b^3)\geq 0\implies b\geq 2+\sqrt{3}$.
 \item\label{propmain.gamma}$s_4(y)=16 + 24 x + 9 x^2 + x^3 + (-33 - 10 x - x^2) y + 6 y^2>0$ and $x\geq 5$ $\implies x<39/5$, where $x=a+b$ and $y=ab$.
 \item\label{propmain.g3a}$h_5(b)=21 + 6 a + a^2 - 8 a b + 4 b^2\leq 0\implies a\geq1+2\sqrt{2}$.
 \item\label{claim.ga}$h_6(b)= \lambda_0+ \lambda_1 b + \lambda_2 b^2 +  \lambda_3b^3 + \lambda_4 b^4>0 \implies a<\frac{11}{3}$, where\begin{align*}
        \lambda_0&=-176 + 96 a - 8 a^2 + 
 4 a^3 + a^4\\
 \lambda_1&=216 - 56 a - 26 a^2 + 4 a^3 - 2 a^4\\
 \lambda_2&=-107 + 
    28 a + 10 a^2 - 2 a^3 + a^4\\
 \lambda_3&=32 - 22 a + 8 a^2 - 
    2 a^3\\
    \lambda_1&=(a-2)^2.
    \end{align*}
 \item\label{claim.gb} $j_7(a)= \lambda_0+ 
 \lambda_1 a  + 
 \lambda_2 a^2  + 
 \lambda_3a^3+ \lambda_4a^4$ and $a\geq\frac{33-\sqrt{129}}{6}>0\implies b<\frac{387}{100}$, where\begin{align*}
     \lambda_0&=-176 + 216 b - 107 b^2 + 32 b^3 + 4 b^4\\
     \lambda_1&=96 - 56 b + 28 b^2 - 22 b^3 - 4 b^4\\
     \lambda_2&=-8 - 26 b + 10 b^2 + 8 b^3 + b^4\\
     \lambda_3&=4 + 4 b - 2 b^2 - 2 b^3\\
     \lambda_4&=(b-1)^2.
 \end{align*}
 \item\label{claim.g'a}$h_8(b)=60 - 32 a + 
 5 a^2 + a^3 + (-67 + 39 a - 6 a^2 - a^3) b + (25 - 17 a + 
    4 a^2) b^2 + (2 - a) b^3\leq 0\implies a\geq\frac{33-\sqrt{129}}{6}$.
     \item\label{abderivative.g1a}$h_9(b)=60 - 32 a + 
 5 a^2 + a^3 + (-67 + 39 a - 6 a^2 - a^3) b + (25 - 17 a + 
    4 a^2) b^2 + (2 - a) b^3> 0\implies a<\frac{372}{100}$.
 \item\label{claim.g'b}$j_{10}(a)=60 + a^3 (1 - b) - 67 b + 25 b^2 + 2 b^3 + a^2 (5 - 6 b + 4 b^2) + 
 a (-32 + 39 b - 17 b^2 - b^3)\leq 0$ and $a<\frac{11}{3}\implies b>\frac{389}{100}$.
  \item\label{abderivative.g1b}$j_{11}(a)=60 + a^3 (1 - b) - 67 b + 25 b^2 + 2 b^3 + a^2 (5 - 6 b + 4 b^2) + 
 a (-32 + 39 b - 17 b^2 - b^3)> 0$ and $a>\frac{11}{3}\implies b<\frac{390}{100}$.
 \end{enumerate}
\end{lem}
\begin{proof}
We prove \ref{propmain.g2'a}, \ref{propmain.g2a}, \ref{propmain.g2b}, \ref{propmain.gamma} and \ref{claim.ga} here.

\ref{propmain.g2'a}. $h_1(b)$ is quadratic in $b$ with negative leading coefficient. In order for the semi-algebraic set $h_1(b)\geq 0$ to be nonempty on the domain, it needs to have non-negative discriminant \[\Delta=-44 + 64 a - 35 a^2 + 2 a^3 + a^4.\]It then follows that $h_1'(4)>0$ since \[h_1'(4)=-30 + 5 a + a^2\leq 0\]implies $1\leq a\leq \frac{\sqrt{145}-5}{2}$, for which $\Delta<0$ by calculating the corresponding Sturm sequence. Therefore, \[h_1(b)\geq 0\implies h_1(4)=-59 + 8 a + 2 a^2\geq 0\implies a\geq \frac{\sqrt{134}-4}{2}.\]

\ref{propmain.g2a}. $h_2(b)$ is cubic in $b$ with leading coefficient $2a-4$. Since the goal is to show $a<2+\sqrt{3}$, we can assume $a>2$ so that $h_2$ has positive leading coefficient. Notice that $h_2(b)<0$ means $a$ has to be less than the largest root of the polynomial. In turn this means if $h_2(a)\geq 0$ and $h_2'(a)>0$, then $h_2''(a)<0$. This is not the case. \begin{align*}
    h_2(a)=-3 (-4 + a) (1 - 4 a + a^2)\geq 0&\implies 2+\sqrt{3}\leq a\leq 4 \\&\implies h_2'(a)=(a-3)^3>0\\&\implies h_2''(a)=4 - 22 a + 6 a^2>0.
\end{align*}Therefore, we must have $h_2(a)<0 \implies a<2+\sqrt{3}$.

\ref{propmain.g2b}. $j_3(a)$ is cubic in $a$ with positive leading coefficient. We consider two cases based on the sign of the discriminant:

(i) $\Delta< 0$. Then \[j_3(a)\geq 0\implies j_3(b)=-3 (-4 + b) (1 - 4 b + b^2)\geq 0\iff 2+\sqrt{3}\leq b\leq 4.\]

(ii) $\Delta\geq 0$. Then $b$ has to be bigger than the smallest root of $j_3$. We first notice that $j_3''(b)=-2-4b<0$. So if $j_3(b)<0$, then we must have $j_3'(b)<0$. On the other hand, $\Delta$ is a polynomial $b$ with degree 10 and $\Delta\geq 0\implies b>3$ by computing the Sturm sequence for the interval $[0,3]$.  But $j_3'(b)=-24 + 21 b - b^3>0$ for $3<b<2+\sqrt{3}$. So in this case we also have $j_3(b)\geq 0\implies b\geq 2+\sqrt{3}$.

\ref{propmain.gamma}. $s_4(y)$ is quadratic in $y$ with positive leading coefficient. Note that $1\leq a\leq b\leq 4$ and $a+b\geq 5$ is equivalent to $4\leq 4x-16\leq y\leq x^2/4\leq 16$. $s_4(y)>0$ implies that if $s_4(x^2/4)\leq 0$, then $s_4(4x-16)> 0$. However,
\[
s_4\left(x^2/4\right)=\frac{1}{8} (-8 + x) (-16 - 26 x - 4 x^2 + x^3)\leq 0\implies x\geq 766/100.
\]
and 
\[
s_4(4x-16)=-(-8 + x) (260 - 57 x + 3 x^2)> 0\implies x<761/100.
\]

Therefore, $s_4(x^2/4)> 0\implies x<\frac{39}{5}$.

\ref{claim.ga}. $h_6$ is quartic with positive leading coefficient. We claim that $h_6(b)>0\implies h_6(a)>0$. We prove this by contradiction. 

Assume $h_6(a)=-(-4 + a)^2 (-1 + a) (-11 + 3 a)\leq 0$. Then we get $\frac{11}{3}\leq a\leq 4$. Direct computation obtains the following:
$h_6(4)<0$, $h_6'(a)<0$, $h_6'''(a)>0$. Now if $h_6(b)$ has complex roots, then $h_6(4)<0$ means $h_6(b)\leq 0$ for all $a\leq b\leq 4$. So $h_6'(b)$ must have three distinct roots $b_1<b_2<b_3$. However, $h_6'(a)<0$ and $h_6'''(a)>0$ imply that $b_2<a<b_3$ and $a$ must therefore be greater than the second largest root of $h_6(b)$. This is again a contradiction since $h_6(4)<0$. 

Thus, $h_6(b)>0\implies h_6(a)>0\implies a<\frac{11}{3}$.
\end{proof}

 \begin{lem}\label{ablems2}
 Let $2\leq a\leq b\leq 4$.
 \begin{enumerate}
  \item\label{abderivative.g4a}$j_{12}(a)=\lambda_0+ 
 \lambda_1 a  + 
 \lambda_2 a^2  + 
 \lambda_3a^3\leq 0\implies a<377/100$, where\begin{align*}
     \lambda_0&=48 - 28 b + 14 b^2 - 11 b^3 - 2 b^4\\
     \lambda_1&=-8 - 26 b + 10 b^2 + 8 b^3 + b^4\\
     \lambda_2&=6 + 6 b - 3 b^2 - 3 b^3\\
     \lambda_3&=2(b-1)^2.
 \end{align*}
 \item\label{abderivative.g4'a}$h_{13}(b)=-32 + 10 a + 3 a^2 + (39 - 12 a - 3 a^2) b + (-17 + 8 a) b^2 - b^3>0\implies a>\frac{15-\sqrt{13}}{3}$.
 \item\label{abderivative.g4ap}$h_{14}(b)== \lambda_0+ \lambda_1 b + \lambda_2 b^2 +  \lambda_3b^3 + \lambda_4 b^4\geq 0\implies a\geq\frac{11}{3}$, where\begin{align*}
     \lambda_0&=48 - 8 a + 6 a^2 + 
 2 a^3\\
     \lambda_1&=-28 - 26 a + 6 a^2 - 4 a^3\\
     \lambda_2&=14 + 10 a - 3 a^2 + 
    2 a^3\\
     \lambda_3&=-11 + 8 a - 3 a^2 \\
     \lambda_4&=-2+a.
 \end{align*}
 \item\label{abderivative.g4b}$h_{15}(b)== \lambda_0+ \lambda_1 b + \lambda_2 b^2 +  \lambda_3b^3 + \lambda_4 b^4\geq 0$ and $a<\frac{372}{100}\implies b>\frac{394}{100}$, where\begin{align*}
     \lambda_0&=48 - 8 a + 6 a^2 + 
 2 a^3\\
     \lambda_1&=-28 - 26 a + 6 a^2 - 4 a^3\\
     \lambda_2&=14 + 10 a - 3 a^2 + 
    2 a^3\\
     \lambda_3&=-11 + 8 a - 3 a^2 \\
     \lambda_4&=-2+a.
 \end{align*}
\item\label{abderivative.g14}$h_{16}(b)=16 + 24 a + 
 9 a^2 + a^3 + (24 - 15 a - 7 a^2 - a^3) b + (9 - 7 a + 
    4 a^2) b^2 + (1 - a) b^3>0\implies a<\frac{384}{100}$.
    \item\label{abderivative.alpha}$h_{17}(b)=53 - 3 a - 
 3 a^2 - a^3 + (-123 + 42 a - 3 a^2) b + (21 - 3 a) b^2 - b^3\leq 0$ and $a\geq\frac{33-\sqrt{129}}{6}\implies a>\frac{388}{100}$.
 \item\label{abderivative.g1'}$s_{18}(x)=\lambda_0 + 
 \lambda_2x^2  + \lambda_1 x + \lambda_3 x^3>0$, where $x=b-a$, $y=a+b$ and \begin{align*}
     \lambda_0&=-271  - 1037 y + 854 y^2 - 226 y^3 + 25 y^4 - y^5\\
     \lambda_1&=-492 + 756 y - 180 y^2 + 12 y^3\\
     \lambda_2&=-378 + 180 y - 18 y^2\\
     \lambda_3&=144.
 \end{align*}
 \item\label{abderivative.g5b}$h_{19}(b)=96 + 8 a - 36 a^2 + 2 a^3 - 
 2 a^4 + (-80 - 22 a + 32 a^2 - 2 a^3 + 2 a^4) b + (46 - 32 a + 
    16 a^2 - 6 a^3) b^2 + (12 - 14 a + 4 a^2) b^3\geq 0\implies a\geq\frac{33-\sqrt{129}}{6}$.
 \item\label{abderivative.g54}$h_{20}(b)=64 + 48 a - 92 a^2 - 6 a^3 - 
 2 a^4 + (96 - 184 a + 110 a^2 + 2 a^4) b + (36 - 18 a - 
    6 a^3) b^2 + (4 - 8 a + 4 a^2) b^3\geq 0\implies a\geq\frac{33-\sqrt{129}}{6}$.
    \item\label{abderivative.g5bn}$h_{21}(b)=96 + 8 a - 36 a^2 + 2 a^3 - 
 2 a^4 + (-80 - 22 a + 32 a^2 - 2 a^3 + 2 a^4) b + (46 - 32 a + 
    16 a^2 - 6 a^3) b^2 + (12 - 14 a + 4 a^2) b^3< 0\implies a<\frac{15}{4}$.
    \item\label{abderivative.g5'b}$h_{22}(b)=-64 + 50 a - 14 a^2 - 2 a^3 + (48 - 48 a + 16 a^2) b + (10 - 6 a) b^2>0\implies a>\frac{387}{100}$.
\end{enumerate}
\end{lem}
\begin{proof} 
We prove \ref{abderivative.g4a} and \ref{abderivative.g1'} here.

\ref{abderivative.g4a}. We show directly $a\geq 377/100\implies j_{10}(a)>0$. We compute directly the discriminant $\Delta\leq 0$ (by using the Sturm sequence on $(2,4)$). Then we easily check that $j_{10}(377/100)>0$ and $j_{10}(b)>0$.

\ref{abderivative.g1'}. We show $s_{18}{x}>0$ for all $0\leq x\leq 2$ and $4\leq y\leq 8$. $s_{18}(x)$ is cubic in $x$ with positive leading coefficient. We first directly compute that the discriminant of $s_{18}'(x)$ is always negative for all $4\leq y\leq 8$, so $s_{18}'(x)>0$. Now we simply check that $s_{18}(0)>0$ for all $4\leq y\leq 8$ and we are done.
\end{proof}
       The next lemma is about semi-algebraic sets who do not belong to different quadrants. In this case, we consider them as one variable polynomials and compare the roots along with the domain bounds on the variable.
\begin{lem}\label{2d}
      Let $a,b$ be such that $2\leq a\leq b\leq 4$ and $14+2b-10a<0$. Define\begin{align*}
          h_1(b)&=60 - 32 a + 
 5 a^2 + a^3 + (-67 + 39 a - 6 a^2 - a^3) b + (25 - 17 a + 
    4 a^2) b^2 + (2 - a) b^3\\
    \begin{split}
        h_2(b)&=48 + 4 a - 
 18 a^2 + a^3 - a^4  -(40 + 11 a - 16 a^2 + a^3 - a^4) b\\ &+ (23 - 
    16 a + 8 a^2 - 3 a^3) b^2 + (6 - 7 a + 2 a^2) b^3
    \end{split}\\
    h_3(b)&=53 - 123 a + 
 21 a^2 - a^3 + (-3 + 42 a - 3 a^2) b + (-3 - 3 a) b^2 - b^3\\
 h_4(b)&=16 - 28 a + 49 a^2 - 3 a^3 + (24 - 58 a - a^2) b + (9 + 3 a) b^2 + b^3\\
 s_5(x)&=16 - 2 y + (26 - 20 y + 2 y^2)x + (29-y)x^2\\
 \begin{split}
 s_6(x)&=271  + 1037 y - 854 y^2 + 226 y^3 - 25 y^4 + y^5+ (-492 + 756 y - 180 y^2 + 12 y^3)x \\&+ 
 (378 - 180 y + 18 y^2)x^2  +144 x^3 .
 \end{split}
      \end{align*}
      \begin{enumerate}
          \item\label{g1pg5n} $h_1(b)>0\implies h_2(b)<0$
 \item\label{alphagamma}$h_3(b)\geq 0$ and $a\geq\frac{33-\sqrt{129}}{6}\implies h_4(b)>0$.
 \item\label{gammag1'}$s_5(x)>0$ and $a\geq\frac{33-\sqrt{129}}{6}\implies s_6(x)<0$, where $x=b-a$ and $y=a+b$.
      \end{enumerate}.
      \end{lem}
      \begin{proof}We prove Statement~\ref{g1pg5n}. \ref{alphagamma} and \ref{gammag1'} are similar.
      
      \ref{g1pg5n}. We start by assuming $h_1(b)>0$ and $h_2(b)\geq 0$, and reach a contradiction. We get $\frac{33-\sqrt{129}}{6}\leq a<\frac{372}{100}$ (Lemma~\ref{ablems}.\ref{abderivative.g1a} and Lemma~\ref{ablems2}.\ref{abderivative.g5b}). Direct computation shows that $h_2(a)<0$, $h_2''(a)>0$ for $\frac{33-\sqrt{129}}{6}\leq a<\frac{372}{100}$. So $a$ has to be greater than the second largest root of $h_2$ (if it exists). Let $b_i$ be the largest root of $h_i$ respectively. We show that $b_1<b_2$ and this is a contradiction. This is equivalent to $h_2(b_1)<0$. We use long division to get\begin{equation}\label{b1p}
          \frac{h_2(b_1)}{2(a-1)}=(-2 + a)  (-66 + 11 a + a^2) - (161 - 101 a + 11 a^2 + a^3) b_1 + (-4 + a)  (-13 + 5 a) b_1^2.
      \end{equation}
      It suffices to show $b_1>b_0$, where $b_0$ is the larger root of the quadratic polynomial~(\ref{b1p}) in variable $b_1$. This is equivalent to $h_1(b_0)>0$, to which we compute \[
      h_1(b_0)=  \frac{6\left(\alpha+\gamma\sqrt{\Delta}\right)}{(4-a)^3(-13+5a)^3},
      \]where\begin{multline*}
          \alpha=(a-3) (292875 - 596625 a + 523563 a^2 - 253410 a^3 \\+ 70828 a^4 - 
   10189 a^5 + 166 a^6 + 175 a^7 - 24 a^8 + a^9)
      \end{multline*}
          and\[
      \gamma=-11041 + 22881 a - 19330 a^2 + 8687 a^3 - 2254 a^4 + 340 a^5 - 
 28 a^6 + a^7.
      \]We can check that $\alpha>0$ by calculating the Sturm sequence of the polynomial for $\frac{33-\sqrt{129}}{6}\leq a<\frac{372}{100}$. Then we compute that \[
      \alpha^2-\gamma^2\Delta=-4 (-4 + a)^4 (-13 + 5 a)^3 (80 - 33 a + 3 a^2) (73 - 90 a + 46 a^2 - 
   11 a^3 + a^4)^2\geq 0
      \] for $\frac{33-\sqrt{129}}{6}\leq a<\frac{372}{100}$ with equality at $a=\frac{33-\sqrt{129}}{6}$. But when $a=\frac{33-\sqrt{129}}{6}$, $\gamma>0$. So $h_1(b_0)>0$ and we are done.
      
      \end{proof}
        \appendix
        \section{Table of all type A triples}
               Given our results, it is easy to classify all the type A triples since for fixed $n_1,n_2$ we just need to find the smallest $n_3$ for $F$ to be positive. Here is a table listing all type A triples $(n_1,n_2,n_3)$ with $n_1\leq n_2\leq n_3$. 
       \bigskip

       \begin{table}[h]
           \centering
                      \caption{Type A triples  $(n_1,n_2,n_3)$}\label{typeA}
           \begin{tabular}{ |c|c|c|c| } 
\hline
$n_1<10$ & $n_2\geq n_1$ & $n_3\geq n_2$\\
\hline
\multirow{12}{4em}{\centering $n_1=10$} & $10\leq n_2
\leq 14$& $n_3\geq n_2$ \\
\cline{2-3}
& $n_2=15$ & $n_3\geq 16$ \\ 
\cline{2-3}
& $n_2=16$ & $n_3\geq 17$ \\ 
\cline{2-3}
& $n_2=17$ & $n_3\geq 19$ \\ 
\cline{2-3}
& $n_2=18$ & $n_3\geq 21$ \\ 
\cline{2-3}
& $n_2=19$ & $n_3\geq 24$ \\ 
\cline{2-3}
& $n_2=20$ & $n_3\geq 27$ \\ 
\cline{2-3}
& $n_2=21$ & $n_3\geq 31$ \\ 
\cline{2-3}
& $n_2=22$ & $n_3\geq 36$ \\ 
\cline{2-3}
& $n_2=23$ & $n_3\geq 44$ \\ 
\cline{2-3}
& $n_2=24$ & $n_3\geq 59$ \\ 
\cline{2-3}
& $n_2=25$ & $n_3\geq 113$ \\ 
\cline{2-3}
\hline
\multirow{7}{4em}{\centering $n_1=11$}
& $n_2=11$ & $n_3\geq 12$ \\ 
\cline{2-3}
& $n_2=12$ & $n_3\geq 14$
\\ \cline{2-3}
& $n_2=13$ & $n_3\geq 16$
\\ \cline{2-3}
& $n_2=14$ & $n_3\geq 19$
\\ \cline{2-3}
& $n_2=15$ & $n_3\geq 23$
\\ \cline{2-3}
& $n_2=16$ & $n_3\geq 31$
\\ \cline{2-3}
& $n_2=17$ & $n_3\geq 49$\\ 
\hline
\multirow{3}{4em}{\centering $n_1=12$}
& $n_2=12$ & $n_3\geq 17$ \\ 
\cline{2-3}
& $n_2=13$ & $n_3\geq 22$
\\ \cline{2-3}
& $n_2=14$ & $n_3\geq 33$\\
\hline
 $n_1=13$ & $n_2=13$ & $n_3\geq 40$\\
\hline
\end{tabular}
       \end{table}

\bigskip
\bibliography{reference}
\bibliographystyle{alpha}
\ContactInfo
  
\end{document}